\RequirePackage{fix-cm}
\RequirePackage{amsthm}
\documentclass[sn-mathphys-num,pdflatex]{sn-jnl}

\usepackage{graphicx}
\usepackage{multirow}
\usepackage{amsmath,amssymb,amsfonts}
\usepackage{mathrsfs}
\usepackage[title]{appendix}
\usepackage{xcolor}
\usepackage{textcomp}
\usepackage{manyfoot}
\usepackage{booktabs}
\usepackage{algorithm}
\usepackage{algorithmicx}
\usepackage{algpseudocode}
\usepackage{listings}
\usepackage{anyfontsize}

\usepackage{hyperref}
\usepackage{lmodern}
\usepackage{enumerate} 
\usepackage{enumitem}

\newtheorem{theorem}{Theorem}
\newtheorem{lemma}{Lemma}

\newtheorem{example}{Example}
\newtheorem{remark}{Remark}
\newcounter{assumcounter}
\newtheorem{assum}[assumcounter]{Assumption}

\newcounter{behavcounter}
\newtheorem{behav}[behavcounter]{Behavior}

\newcommand{\R}{\mathbb{R}}
\newcommand{\N}{\mathbb{N}}
\newcommand{\C}{\mathcal{C}}
\newcommand{\calN}{\mathcal{N}}
\newcommand{\calL}{\mathcal{L}}

\newcommand{\eps}{\varepsilon}
\newcommand{\spd}{sym.\ pos.\ def.\ }
\newcommand{\id}{\mathbf{I}}

\DeclareMathOperator*{\conv}{conv}

\DeclareMathOperator*{\spn}{span}

\begin{document}

\title{Technical results on the convergence of quasi-Newton methods for nonsmooth optimization}

\author*[1]{\fnm{Bennet} \sur{Gebken}}\email{bennet.gebken@cit.tum.de}

\affil*[1]{\orgdiv{Department of Mathematics}, \orgname{Technical University of Munich}, \orgaddress{\street{Boltzmannstr. 3}, \city{Garching b. München}, \postcode{85748}, \country{Germany}}}

\abstract{
    It is well-known by now that the BFGS method is an effective method for minimizing nonsmooth functions. However, despite its popularity, theoretical convergence results are almost non-existent. One of the difficulties when analyzing the nonsmooth case is the fact that the secant equation forces certain eigenvalues of the quasi-Newton matrix to vanish, which is a behavior that has not yet been fully analyzed. In this article, we show what kind of behavior of the eigenvalues would be sufficient to be able to prove the convergence for piecewise differentiable functions. More precisely, we derive assumptions on the behavior from numerical experiments and then prove criticality of the limit under these assumptions. Furthermore, we show how quasi-Newton methods are able to explore the piecewise structure. While we do not prove that the observed behavior of the eigenvalues actually occurs, we believe that these results still give insight, and a certain intuition, for the convergence for nonsmooth functions.
}

\keywords{Nonsmooth optimization, Nonconvex optimization, quasi-Newton, BFGS}
\pacs[MSC Classification]{90C30, 65K05, 65K10}

\maketitle

\section{Introduction} \label{sec:introduction}

    Quasi-Newton methods \cite{D1959,NW2006} are among the most popular methods for minimizing smooth functions. Their core idea is to replace the inverse Hessian matrix in Newton's method by a \spd approximation $H_{k+1} \in \R^{n \times n}$, the \emph{quasi-Newton matrix},
    satisfying the \emph{(inverse) secant equation}
    \begin{align} \label{eq:secant_equation}
        H_{k+1}(\nabla f(x^{k+1}) - \nabla f(x^k)) = x^{k+1} - x^k,
    \end{align}
    and to then use $p^{k+1} = -H_{k+1} \nabla f(x^{k+1})$ as a search direction at $x^{k+1}$. Despite only using first-order derivatives, these methods are able to achieve superlinear convergence (see, e.g., \cite{NW2006}, Theorem 6.6), which makes them highly desirable for smooth optimization. However, surprisingly, it can be observed empirically that quasi-Newton methods, specifically the BFGS method, also work well for nonsmooth optimization, typically converging with a linear rate. This is surprising, since Newton's method, which these methods arguably try to mimic, fails even on simple nonsmooth functions (like convex piecewise linear functions, where it fails like gradient descent \cite{AO2019}).
    In particular, quasi-Newton methods do not contain any classic technique for handling nonsmoothness, like bundling \cite{BKM2014} or gradient sampling \cite{BLO2005}.
    Their convergence was first commented on by Lemaréchal in \cite{L1982} and was popularized by Lewis and Overton in \cite{LO2013}, who posed a challenge to provide the theoretical reason for it (\cite{LO2013}, Challenge 7.1).

    Despite the popularity of quasi-Newton methods for nonsmooth optimization, there are only few theoretical convergence results: In \cite{LO2013}, Theorem 3.2, the convergence of a general quasi-Newton method with exact line search applied to the Euclidean norm function $f : \R^2 \rightarrow \R$, $x \mapsto \| x \|$, on $\R^2$ is proven. Furthermore, in Section 5.1, convergence with an inexact Wolfe line search is proven for the absolute value function $f : \R \rightarrow \R$, $x \mapsto | x |$. In \cite{GL2018}, Corollary 4.2, the convergence of the BFGS method with Wolfe step length is proven for the Euclidean norm on $\R^n$ for arbitrary $n$. In \cite{LZ2015}, Proposition 4.2, and \cite{XW2017}, Theorem IV.1 and Remark IV.2, it is shown that for certain unbounded below, piecewise linear functions, the BFGS method with inexact Wolfe line search does not converge to a non-critical point. (Related results for the limited-memory BFGS method are proven in \cite{AO2020,AO2021}.) However, even for simple nonsmooth functions like $f : \R^2 \rightarrow \R$, $x \mapsto x_1^2 + |x_2|$ from \cite{GL2018}, the convergence of the BFGS method is not yet understood.

    In the smooth case, the standard convergence theory for the BFGS method (see, e.g., \cite{NW2006}, Theorem 6.5) is based on estimates for the smallest and largest eigenvalue of $(H_k)_k$. In the nonsmooth case, for a sequence $(x^k)_k$ with limit $\bar{x}$, the difference of iterates on the right-hand side of the secant equation \eqref{eq:secant_equation} vanishes, but the difference of gradients on the left-hand side may not vanish. This means that \eqref{eq:secant_equation} forces certain eigenvalues of $(H_k)_k$ to vanish (with the corresponding eigenvectors being the discontinuous jumps of $\nabla f$ locally around $\bar{x}$). As a result, it is unclear how the convergence theory from the smooth case can be generalized. Furthermore, the condition number of $(H_k)_k$ becomes unbounded, which causes numerical issues due to limited machine precision.
 
    In this article, we skip the theoretical analysis of the eigenvalues of $(H_k)_k$ and instead show what behavior of $(H_k)_k$ would be sufficient to prove certain convergence results of quasi-Newton methods.
    We restrict ourselves to a well-behaved subclass of the class of piecewise differentiable functions \cite{S2012}, since for these functions, the above-mentioned discontinuous jumps of $\nabla f$ simply correspond to jumps between the different areas in which $f$ is smooth. The first of two main results (Theorem \ref{thm:criticality}) shows that if the generated sequence $(x^k)_k$ has a limit $\bar{x}$ and if only those eigenvalues that are forced to vanish by \eqref{eq:secant_equation} vanish (Behavior \ref{behav:B1}), then $\bar{x}$ is Clarke critical. The second main result (Theorem \ref{thm:exploration}) shows that if the initial point $x^0$ is close enough to the global minimum and small eigenvalues of $(H_k)_k$ stay small for a certain number of iterations (Behavior \ref{behav:B2}), then $(x^k)_k$ visits each of the $m$ smooth pieces of $f$ exactly once in the first $m$ iterates. This gives insight into the way in which quasi-Newton methods are able to explore the structure of nonsmooth functions, and can also be used to explain the behavior that occurs when restarts are introduced.
    From a technical point of view, the theory in this article is based around showing that in certain situations, the gradients of the selection functions of $f$ are contained in the kernels of accumulation points of $(H_k)_k$, which connects the eigenvalues and eigenvectors of $(H_k)_k$ to derivative information at the limit $\bar{x}$.
    
    We emphasize that our two main results fully rely on the behavioral assumptions \ref{behav:B1} and \ref{behav:B2} holding, which we \emph{do not} prove in this article. While numerical experiments (see Example \ref{example:criticality} and Example \ref{example:exploration}) suggest that they generically hold for a relatively general class of functions, there does not appear to be a way to actually prove them. (As discussed in Remark \ref{remark:challenge} below, the assumption \ref{behav:B1} is closely related to 4.\ in  Challenge 7.1 in \cite{LO2013}. See also the discussion in Section \ref{sec:outlook}.) However, we believe that the theory in this article still gives a certain intuition for why quasi-Newton methods work for nonsmooth functions. In particular, it shows how the secant equation ``encodes'' nonsmooth information into the quasi-Newton matrix, and how this yields descent directions with sufficient decrease without any explicit bundling or gradient sampling (and without the need to solve subproblems). Furthermore, \ref{behav:B1} and \ref{behav:B2} may serve as waypoints when developing a convergence theory from the ground up.

    Matlab scripts for the reproduction of all numerical experiments shown in this article are available at \url{https://github.com/b-gebken/Nonsmooth-BFGS-experiments}. To avoid any issues related to machine precision, we use Matlab's variable precision arithmetic (\verb+vpa+) with $500$ significant digits for some experiments.

    The rest of this article is organized as follows: In Section \ref{sec:preliminaries} we introduce the basics of quasi-Newton methods and piecewise differentiable functions. In Section \ref{sec:limit_critical} we first discuss the required assumptions on the asymptotic behavior of $(H_k)_k$ for $k \rightarrow \infty$ and afterwards prove the first main result, regarding the criticality of the limit. In Section \ref{sec:explore_structure} we discuss the non-asymptotic behavior of $(H_k)_k$ close to the minimum and then prove the second main result, regarding the exploration of the nonsmooth structure. Here, we also briefly consider the behavior of the BFGS methods with restarts on nonsmooth functions. Finally, we discuss open questions and possible directions for future research in Section \ref{sec:outlook}. 

\section{Preliminaries} \label{sec:preliminaries}

    In this section, we introduce the basics of quasi-Newton methods and piecewise differentiable functions. For more detailed introductions, we refer to \cite{NW2006}, Chapter 6, and \cite{S2012}, Chapter 4, respectively. We will also use basic results about affine independence throughout the article, which can be found, e.g., in \cite{B1983}, Section \S 1.

    \subsection{Quasi-Newton methods}

        Consider a function $f : \R^n \rightarrow \R$ and denote the set of points at which $f$ is not differentiable by $\Omega$. For $x \notin \Omega$, $p \in \R^n$ with $\nabla f(x)^\top p < 0$, and $c_1, c_2 \in (0,1)$, $c_1 < c_2$, a step length $t > 0$ satisfies the \emph{Wolfe conditions}, if $x + t p \notin \Omega$ and
        \begin{align}
            f(x + t p) &\leq f(x) + c_1 t \nabla f(x)^\top p, \label{eq:Wolfe_1} \tag{W1} \\
            \nabla f(x + t p)^\top p &\geq c_2 \nabla f(x)^\top p. \label{eq:Wolfe_2} \tag{W2}
        \end{align}
        By \cite{LO2013}, Theorem 4.5, if $f$ is locally Lipschitz continuous, then a step length satisfying the Wolfe conditions exists. Computing the next iterate $x^{k+1} = x^k + t p^k$ via a Wolfe step length guarantees that a \spd matrix $H_{k+1}$ satisfying \eqref{eq:secant_equation} exists (cf.\ (6.8) in \cite{NW2006}). This shows that the general quasi-Newton method (including a differentiability check) denoted in Alg.\ \ref{algo:QN} is well-defined.
        \begin{algorithm} 
            \caption{Quasi-Newton method}
            \label{algo:QN}
            \begin{algorithmic}[1] 
                \Require Initial point $x^0 \in \R^n \setminus \Omega$, initial \spd matrix $H_0 \in \R^{n \times n}$, Wolfe parameters $c_1, c_2 \in (0,1)$, $c_1 < c_2$.
                \For{$k = 0, 1, \dots$}
                    \State If $\nabla f(x^{k}) = 0$ then stop.
                    \State Set $p^k = - H_k \nabla f(x^k)$.
                    \State Set $x^{k+1} = x^k + t_k p^k$, where $t_k$ satisfies the Wolfe conditions \eqref{eq:Wolfe_1} and \eqref{eq:Wolfe_2}.
                    \State If $f$ is not differentiable at $x^{k+1}$ then stop. \label{state:breakdown}
                    \State For $y^k = \nabla f(x^{k+1}) - \nabla f(x^k)$ and $s^k = x^{k+1} - x^k$, compute \label{state:QN_update}
                    \Statex \hspace{\algorithmicindent}a \spd matrix $H_{k+1}$ with $H_{k+1} y^k = s^k$.
                \EndFor
            \end{algorithmic}
        \end{algorithm}
        As in \cite{LO2013}, if the algorithm stops in Step \ref{state:breakdown} with $x^{k+1} \in \Omega$, then we say that it \emph{breaks down}. For $k \in \N \cup \{ 0 \}$ we denote the eigenvalues of $H_k$ by $0 < \lambda_1^k \leq \dots \leq \lambda_n^k$. There are many ways for computing the matrix $H_{k+1}$ in Step \ref{state:QN_update} of Alg.\ \ref{algo:QN}. For all numerical experiments in this article, we use the \emph{BFGS update}, where
        \begin{align} \label{eq:BFGS_update} 
            H_{k+1} = V_k H_k V_k^\top + \frac{s^k (s^k)^\top}{(s^k)^\top y^k} \text{ for } V_k = \id - \frac{s^k (y^k)^\top}{(s^k)^\top y^k},
        \end{align}
        with $\id \in \R^{n \times n}$ denoting the identity matrix. Throughout the article, the term \emph{BFGS method} refers to the method that results from using \eqref{eq:BFGS_update} for Step \ref{state:QN_update} of Alg.\ \ref{algo:QN}. For computing the Wolfe step length in the numerical experiments, we use Alg.\ 4.6 from \cite{LO2013}. (However, for our theoretical results, the specific way in which the step length is computed does not matter.) Finally, as in \cite{LO2013}, we do not actually check for differentiability in practice, since there is no reliable way to do so in a numerical setting.

    \subsection{Piecewise differentiable functions}

        A function $f : \R^n \rightarrow \R$ is called a \emph{continuous selection} of the functions $f_i : \R^n \rightarrow \R$, $i \in I := \{1,\dots,m\}$, $m \in \N$, on $\R^n$, if $f$ is continuous and 
        \begin{align*}
            f(x) \in \{ f_i(x) : i \in I \} \quad \forall x \in \R^n.
        \end{align*}
        The functions $f_i$ are referred to as \emph{selection functions} of $f$. A selection function $f_i$ is called \emph{active} at $x$ if $f(x) = f_i(x)$. The \emph{active set} at $x$ is defined by $I(x) := \{ i \in I : f(x) = f_i(x) \}$ and the \emph{essentially active set} is defined by
        \begin{align*}
            I_e(x) := \{ i \in I : x \in \mathrm{cl}(\mathrm{int}(\{y \in \R^n : f(y) = f_i(y)\})) \}.
        \end{align*}
        In the following, let $f$ be a continuous selection of $\C^1$-functions. By \cite{U2002}, Proposition 2.24, for every $x \notin \Omega$, the set
        \begin{align} \label{eq:def_Ag}
            I_g(x) := \{ i \in I_e(x) : \nabla f(x) = \nabla f_i(x) \}
        \end{align}
        is non-empty. 
        By \cite{S2012}, Corollary 4.1.1, $f$ is locally Lipschitz continuous. By \cite{S2012}, Proposition 4.3.1, the Clarke subdifferential \cite{C1990} of $f$ at $x$ is given by $\partial f(x) = \conv(\{ \nabla f_i(x) : i \in I_e(x) \})$. We say that $x$ is a \emph{(Clarke) critical point} of $f$ if $0 \in \partial f(x)$, which is a necessary condition for local optimality (cf.\ \cite{BKM2014}, Theorem 3.17). Finally, if $f$ is a function such that for every $x \in \R^n$, there is an open neighborhood $U \subseteq \R^n$ of $x$ such that the restriction $f|_U$ is a continuous selection of $\C^1$-functions, then $f$ is called \emph{piecewise differentiable}.

\section{Criticality of the limit} \label{sec:limit_critical}

    In this section, we analyze the criticality of the limit (if it exists) of a sequence generated by Alg.\ \ref{algo:QN} for a piecewise differentiable function. Since criticality is a local property, it is sufficient to consider continuous selections of $\C^1$-functions (with a fixed set of selection functions). Since we later want to use a result that generalizes part of the proof of Zoutendijk's theorem (see, e.g., \cite{NW2006}, Theorem 3.2), we further have to assume that these $\C^1$-functions have a locally Lipschitz continuous gradient. More formally, we consider the following class of functions:
    \begin{assum} \label{assum:A1}
        The function $f : \R^n \rightarrow \R$ is a continuous selection of $\C^1$-functions $f_i$, $i \in I = \{1,\dots,m\}$, whose gradients are locally Lipschitz continuous.
    \end{assum}

    In the following, we introduce the assumptions we make for the behavior of Alg.\ \ref{algo:QN} to be able to analyze its convergence. First of all, clearly, it is only relevant to consider the convergence if the algorithm does not break down and the generated sequence $(x^k)_k$ is infinite. Furthermore, we have to assume that $(x^k)_k$ has a limit $\bar{x}$, since even when $f$ is smooth (with $m = 1$), there are examples where $(x^k)_k$ cycles between non-critical points of $f$ \cite{D2002,D2013}. In addition, for ease of notation, we assume that every selection function is active infinitely many times along $(x^k)_k$, as otherwise, it would suffice to consider a continuous selection of a subset of $\{ f_i : i \in I \}$. Regarding the eigenvalues of $(H_k)_k$, note that if $i \in I_g(x^k)$ and $j \in I_g(x^{k+1})$ with $i \neq j$, then for large $k$, the vector $\nabla f_j(x^{k+1}) - \nabla f_i(x^k)$ is mapped to ``almost zero'' by $H_{k+1}$ due to the secant equation \eqref{eq:secant_equation}. For $k \rightarrow \infty$, these vectors belong to the set
    \begin{align} \label{eq:def_calN}
        \calN(\bar{x}) := \spn(\{ \nabla f_i(\bar{x}) - \nabla f_1(\bar{x}) : i \in \{2,\dots,m\} \}).
    \end{align}
    (Note that this definition is independent of the choice of the fixed index $1$.) If the quasi-Newton update is done in a way such that $H_{k+1}$ not only satisfies the current secant equation \eqref{eq:secant_equation}, but also ``memorizes'' previous ones, then $\calN(\bar{x})$ would approximately belong to the kernel of $H_k$ for large $k$. In particular, the number of approximately zero eigenvalues of $H_k$ would be at least $\dim(\calN(\bar{x}))$. The following numerical experiment suggests that for the BFGS update, this is indeed the case, with the number of approximately zero eigenvalues being exactly $\dim(\calN(\bar{x}))$:
    \begin{example} \label{example:criticality}
        For $m \in \{1,\dots,n+1\}$ and $I = \{1,\dots,m\}$, consider the strongly convex function
        \begin{align*}
            f : \R^n \rightarrow \R,
            \quad
            x \mapsto \max_{i \in I} \left( g_i^\top x + \frac{1}{2} x^\top M_i x + \frac{d_i}{24} \| x \|^4 \right)
        \end{align*}
        from \cite{LW2019}, p.\ 26, where $d_i > 0$ for all $i \in I$, $M_i \in \R^{n \times n}$ is \spd for all $i \in I$, and the vectors $g_i \in \R^n$, $i \in I$, are affinely independent with $0 \in \conv(\{ g_i : i \in I \})$. The global minimal point of this function is $x^* = 0 \in \R^n$ with minimal value $f(x^*) = 0$, and we have $\dim(\calN(x^*)) = m-1$.
        
        We generate $10$ random instances of this function with $n = 10$ and $m = 6$, and apply $1000$ iterations of the BFGS method with Wolfe parameters $c_1 = 10^{-4}$ and $c_2 = 0.5$ (the default values in the \emph{HANSO}\footnote{\url{https://cs.nyu.edu/~overton/software/hanso/}} software package), a random initial point $x^0$, and a random initial matrix $H_0$ for every run. (For details on the random generation, see the code that is referenced in Section \ref{sec:introduction}.) The results are computed with $500$ significant digits via Matlab's variable-precision arithmetic.
        \begin{figure}
            \centering
            \parbox[b]{0.32\textwidth}{
                \centering 
                \includegraphics[width=0.32\textwidth]{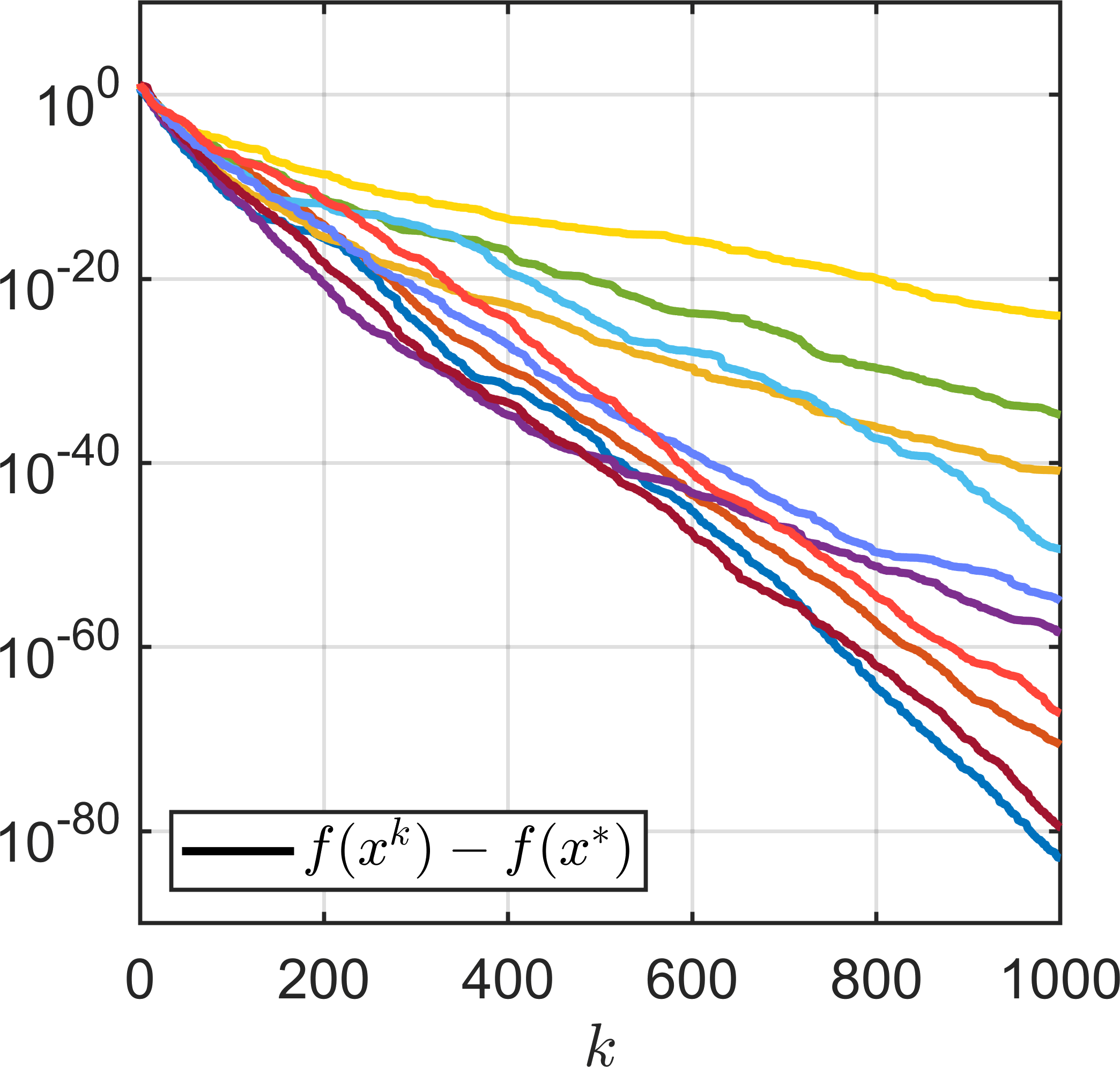}\\
                (a)
    		}
            \parbox[b]{0.32\textwidth}{
                \centering 
                \includegraphics[width=0.32\textwidth]{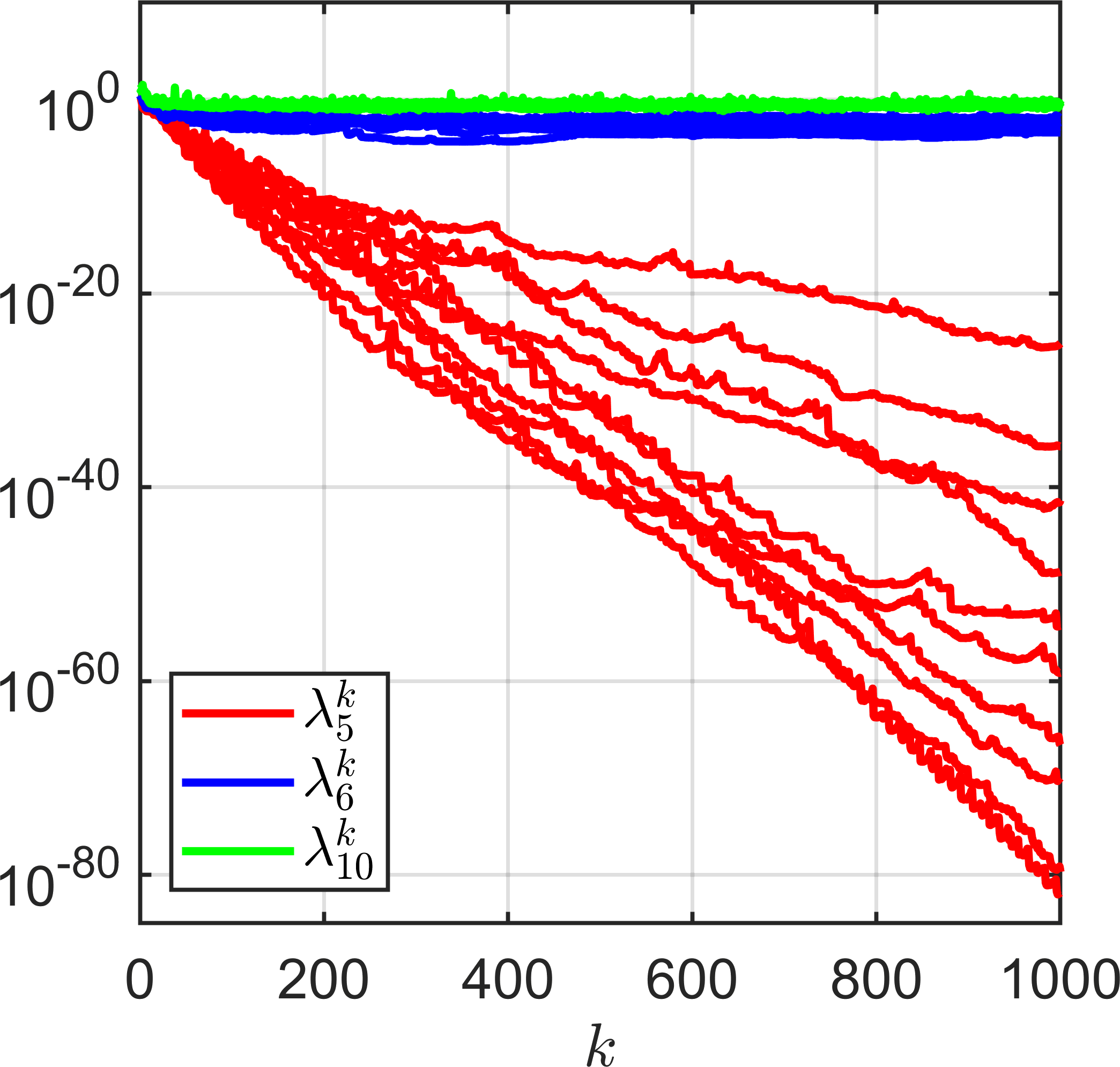}\\
                (b)
    		}
            \parbox[b]{0.32\textwidth}{
                \centering 
                \includegraphics[width=0.32\textwidth]{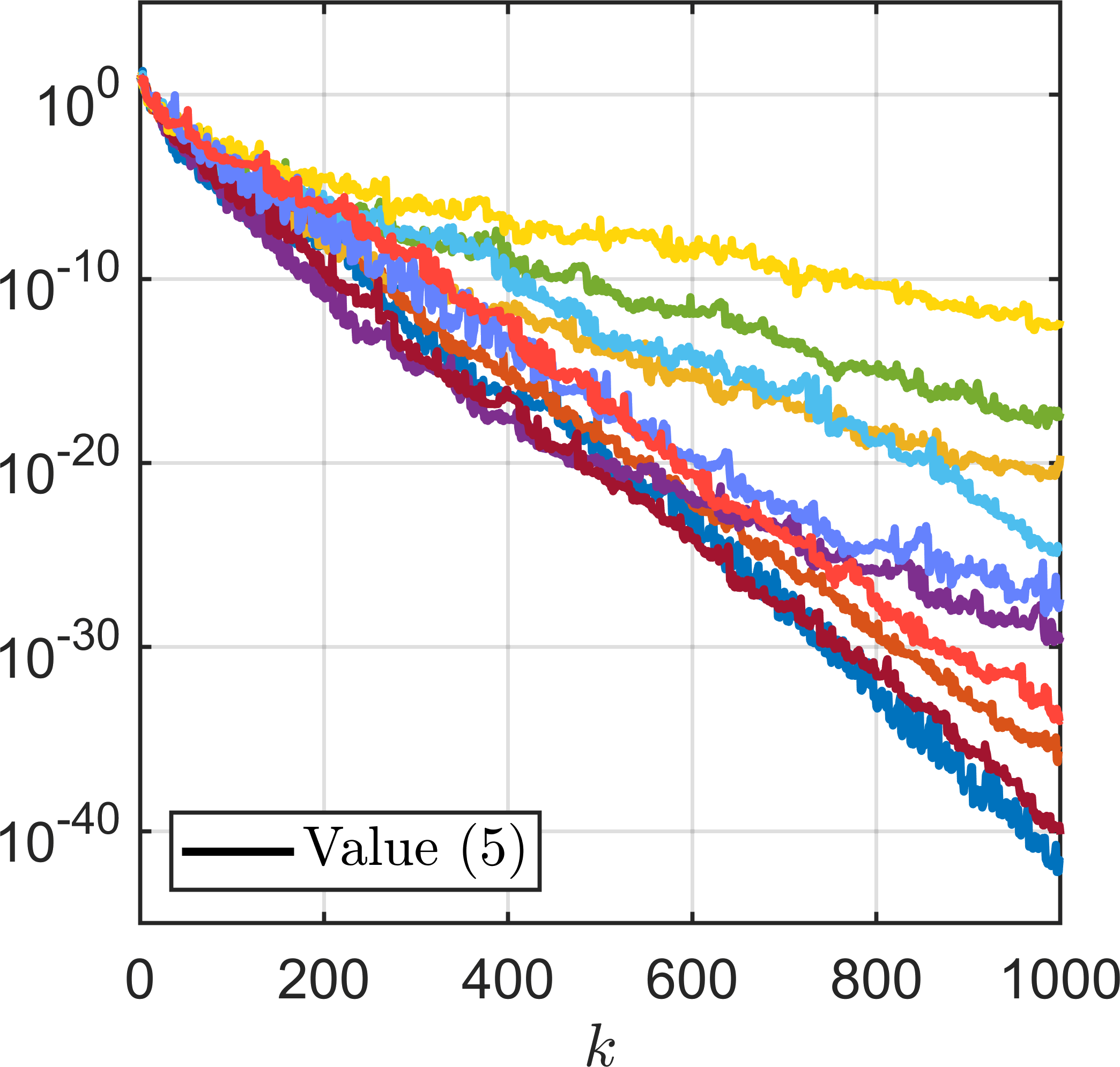}\\
                (c)
    		}
            \caption{(a) The distance of $f(x^k)$ to $f(x^*)$ in Example \ref{example:criticality}. Each color corresponds to one run of the BFGS method for one problem instance. (b) The eigenvalues $\lambda_5^k$, $\lambda_6^k$, and $\lambda_{10}^k$ of $H_k$ for each run. (c) The value \eqref{eq:secant_memory_numerically} for each run, with the same coloring as in (a).}
            \label{fig:example_criticality}
        \end{figure}
        Fig.\ \ref{fig:example_criticality}(a) shows the distance of $f(x^k)$ to the minimal value $0$, where we see the expected (roughly) R-linear rate of convergence. Fig.\ \ref{fig:example_criticality}(b) shows an eigenvalue gap from $\lambda^k_5$ to $\lambda^k_6$, with $\lambda^k_5$ vanishing and $\lambda^k_6$ being bounded away from zero. Furthermore, the largest eigenvalue $\lambda_{10}^k$ appears to be bounded above. Finally, for every $k$, Fig.\ \ref{fig:example_criticality}(c) shows the value
        \begin{align} \label{eq:secant_memory_numerically}
            \max_{i \in \{2,\dots,m\}} \| H_k (\nabla f_i(\bar{x}) - \nabla f_1(\bar{x})) \|
        \end{align}
        with $\bar{x} = 0 \in \R^n$, which appears to vanish as $k$ increases.
    \end{example}

    More formally, the above discussion motivates the following assumptions on the behavior of Alg.\ \ref{algo:QN}:
    \begin{behav} \label{behav:B1}
        Assume that Alg.\ \ref{algo:QN} is applied to a function $f$ satisfying \ref{assum:A1} and generates a sequence $(x^k)_k$ and corresponding quasi-Newton matrices $(H_k)_k$ such that
        \begin{enumerate}[leftmargin=1.3cm,label=(B1.\arabic*)]
            \item \label{enum:B1_1} $(x^k)_k$ is an infinite sequence, i.e., Alg.\ \ref{algo:QN} did not break down and $\nabla f(x^k) \neq 0$ for all $k \in \N$,
            \item \label{enum:B1_2} $(x^k)_k$ has a limit $\bar{x} \in \R^n$,
            \item \label{enum:B1_3} for all $i \in I$, the set $\{ k \in \N : i \in I_g(x^k) \}$ is infinite,
            \item \label{enum:B1_4} there are $\sigma_L, \sigma_U \in \R^{>0}$ such that 
            \begin{align*}
                \lambda_j^k \in [\sigma_L,\sigma_U] \quad \forall j \in \{\dim(\calN(\bar{x}))+1,\dots,n\}, k \in \N,
            \end{align*}
            \item \label{enum:B1_5} it holds
            \begin{align*}
                \lim_{k \rightarrow \infty} H_k (\nabla f_i(\bar{x}) - \nabla f_1(\bar{x})) = 0 \quad \forall i \in \{2,\dots,m\}.
            \end{align*}
        \end{enumerate}
    \end{behav}

    In the following remark, we briefly discuss \ref{behav:B1} in light of Challenge 7.1 in \cite{LO2013}:

    \begin{remark} \label{remark:challenge}
        The assumption \ref{enum:B1_1} corresponds to 1.\ in Challenge 7.1 of \cite{LO2013}. If $f$ is a convex max-function with minimum $x^*$, then in terms of $\mathcal{V}\mathcal{U}$\emph{-decomposition} \cite{LS2020}, the set $\calN(x^*)$ is the $\mathcal{V}$-space (cf.\ \cite{MS1999}, Proposition 1). Alternatively, in terms of \emph{partial smoothness} \cite{L2002}, $\calN(x^*)$ is the normal space $N_{\mathcal{M}}(x^*)$ (cf.\ \cite{L2002}, Theorem 6.1). As such, the assumptions \ref{enum:B1_4} and \ref{enum:B1_5} are closely related to 4.\ in the challenge of \cite{LO2013}. So roughly speaking, in case $f$ is a convex max-function, we are assuming that 1.\ and 4.\ of this challenge hold and analyze 2. with $(x^k)_k$ having a limit. However, note that we are in a deterministic setting, whereas the challenge is posed in a stochastic way.
    \end{remark}

    Our strategy for analyzing the criticality of $\bar{x}$ is based on the intermediate result that $\nabla f_i(\bar{x}) \in \calN(\bar{x})$ for all $i \in I$ (Lemma \ref{lem:intermediate_result} below). Since $\dim(\calN(\bar{x})) \leq m-1$, this implies that the vectors $\nabla f_i(\bar{x})$, $i \in I = \{1,\dots,m\}$, are linearly dependent. While this is merely a (relatively weak) necessary condition for criticality for functions satisfying \ref{assum:A1}, we later consider a subclass of these functions for which it is sufficient (see \ref{assum:A2} below).

    To first prove the intermediate result, we require three technical lemmas. The first one is concerned with the decrease of \emph{all} selection functions along the search direction $p^k$ at $x^k$, not just the active selective function:
    \begin{lemma} \label{lem:common_descent}
        Assume that $f$ satisfies \ref{assum:A1} and that \ref{behav:B1} holds. Then
        \begin{align*}
            \lim_{k \rightarrow \infty} \nabla f_i(x^k)^\top p^k - \nabla f_1(x^k)^\top p^k = 0
            \quad \forall i \in \{2,\dots,m\}.
        \end{align*}
    \end{lemma}
    \begin{proof}
        For $k \in \N$ let $i_k \in I_g(x^k)$. Let $i \in \{2,\dots,m\}$. By definition of $p^k$ we have
        \begin{align*}
            \nabla f_i(x^k)^\top p^k - \nabla f_1(x^k)^\top p^k
            = -(\nabla f_i(x^k) - \nabla f_1(x^k))^\top H_k \nabla f_{i_k}(x^k).
        \end{align*}
        Since $x^k \rightarrow \bar{x}$ (by \ref{enum:B1_2}), $(H_k)_k$ is bounded (w.r.t.\ the spectral norm, by \ref{enum:B1_4}), and the selection functions are $\C^1$ (by \ref{assum:A1}), \ref{enum:B1_5} implies 
        \begin{align*}
            &(\nabla f_i(x^k) - \nabla f_1(x^k))^\top H_k \\
            &= (\nabla f_i(\bar{x}) - \nabla f_1(\bar{x}))^\top H_k + (\nabla f_i(x^k) - \nabla f_1(x^k) - (\nabla f_i(\bar{x}) - \nabla f_1(\bar{x})))^\top H_k
            \rightarrow 0
        \end{align*}
        as $k \rightarrow \infty$, completing the proof.
    \end{proof}

    Lemma \ref{lem:common_descent} shows that all selection functions have approximately the same directional derivative along the search direction for large $k$. 
    The second lemma derives a formula for a lower bound for step lengths satisfying the Wolfe conditions \eqref{eq:Wolfe_1} and \eqref{eq:Wolfe_2}. In words, it shows that if $p^k$ is not only a descent direction for the selection function that is active at $x^k$, but also yields sufficient decrease for the selection function that is active at $x^k + t_k p^k$, then there is a lower bound for $t_k$. It generalizes the lower bound for Wolfe step lengths that is derived in the proof of Theorem 3.2 (Zoutendijk's theorem) in \cite{NW2006} (cf.\ the third inequality in that proof).
    \begin{lemma} \label{lem:lower_bound_Wolfe}
        Assume that $f$ satisfies \ref{assum:A1} and let $x \notin \Omega$. Let $c_2 \in (0,1)$ and $p \in \R^n$ with $\nabla f(x)^\top p < 0$. Let $t$ be a step length satisfying the second Wolfe condition \eqref{eq:Wolfe_2}. For $i \in I_g(x + t p)$ let $L$ be a Lipschitz constant of $\nabla f_i$ on $\conv(\{x,x + tp\})$. Then
        \begin{align*}
            t \geq 
            \frac{1}{L \| p \|^2}
            \left( 
            -(1 - c_2) \nabla f(x)^\top p
            +
            (\nabla f(x) - \nabla f_{i}(x))^\top p
            \right).
        \end{align*}
    \end{lemma}
    \begin{proof}
        By the second Wolfe condition \eqref{eq:Wolfe_2} it holds $\nabla f_{i}(x + tp)^\top p \geq c_2 \nabla f(x)^\top p$. Subtracting $ \nabla f_{i}(x)^\top p$ on both sides yields
        \begin{equation} \label{eq:proof_lem_lower_bound_Wolfe_1}
            \begin{aligned}
                (\nabla f_{i}(x + tp) - \nabla f_{i}(x))^\top p
                &\geq (c_2 \nabla f(x) - \nabla f_{i}(x))^\top p \\
                &= (c_2 \nabla f(x) - \nabla f(x) + \nabla f(x) - \nabla f_{i}(x))^\top p \\
                &= -(1 - c_2) \nabla f(x)^\top p + (\nabla f(x) - \nabla f_{i}(x))^\top p.
            \end{aligned}
        \end{equation}
        Since $\nabla f_{i}$ is locally Lipschitz continuous by \ref{assum:A1}, the left-hand side of \eqref{eq:proof_lem_lower_bound_Wolfe_1} satisfies 
        \begin{align} \label{eq:proof_lem_lower_bound_Wolfe_2}
            (\nabla f_{i}(x + tp) - \nabla f_{i}(x))^\top p
            \leq \| \nabla f_{i}(x + tp) - \nabla f_{i}(x) \| \| p \|
            \leq t L \| p \|^2
        \end{align}
        for a Lipschitz constant of $\nabla f_i$ on $\conv(\{x,x + tp\})$. Combining \eqref{eq:proof_lem_lower_bound_Wolfe_1} and \eqref{eq:proof_lem_lower_bound_Wolfe_2} completes the proof.
    \end{proof}

    Finally, the third lemma shows that $\ker(\bar{H}) = \calN(\bar{x})$ for accumulation points $\bar{H}$ of $(H_k)_k$:
    \begin{lemma} \label{lem:ker_H_subsequence}
        If $f$ satisfies \ref{assum:A1} and \ref{behav:B1} holds, then $(H_k)_k$ has an accumulation point. Furthermore, for all accumulation points $\bar{H}$ of $(H_k)_k$, it holds $\ker(\bar{H}) = \calN(\bar{x})$.
    \end{lemma}
    \begin{proof}
        By \ref{enum:B1_4} the spectral norm of $H_k$ is bounded above by $\sigma_U$ for all $k \in \N$. Thus $(H_k)_k$ must have an accumulation point $\bar{H}$. Let $(k_l)_l \subseteq \N$ be a strictly increasing, infinite sequence with $H_{k_l} \rightarrow \bar{H}$ for $l \rightarrow \infty$. Since $\lambda_j^k \geq \sigma_L$ for all $j \in \{\dim(\calN(\bar{x}))+1,\dots,n\}$, $k \in \N$, we have $\dim(\ker(\bar{H})) \leq \dim(\calN(\bar{x}))$ due to continuity of eigenvalues (see, e.g., \cite{A2013}, Theorem 5.2.2). Thus, it suffices to show that $\calN(\bar{x}) \subseteq \ker(\bar{H})$. To this end, let $v \in \calN(\bar{x})$. Then there are $\alpha_i \in \R$, $i \in \{2,\dots,m\}$, such that $v = \sum_{i = 2}^m \alpha_i (\nabla f_i(\bar{x}) - \nabla f_1(\bar{x}))$. By \ref{enum:B1_5}, we obtain
        \begin{align*}
            \bar{H} v
            = \lim_{l \rightarrow \infty} H_{k_l} v
            = \lim_{l \rightarrow \infty} \sum_{i = 2}^m \alpha_i H_{k_l} (\nabla f_i(\bar{x}) - \nabla f_1(\bar{x}))
            = 0,
        \end{align*}
        so $v \in \ker(\bar{H})$, completing the proof. 
    \end{proof}

    Combination of Lemma \ref{lem:common_descent}, Lemma \ref{lem:lower_bound_Wolfe}, and Lemma \ref{lem:ker_H_subsequence} allow us to prove the intermediate result:

    \begin{lemma} \label{lem:intermediate_result}
        If $f$ satisfies \ref{assum:A1} and \ref{behav:B1} holds, then $I_e(\bar{x}) = I = \{1,\dots,m\}$ and
        \begin{align} \label{eq:intermediate_result}
            \nabla f_i(\bar{x}) \in \calN(\bar{x}) \quad \forall i \in I.
        \end{align}
        In particular, the vectors $\nabla f_i(\bar{x})$, $i \in I$, are linearly dependent.
    \end{lemma}
    \begin{proof}
        The equality $I_e(\bar{x}) = \{1,\dots,m\}$ follows from \ref{enum:B1_3} and the fact that $I_g(x^j) \subseteq I_e(x^j)$ for all $j \in \N$ (cf.\ \eqref{eq:def_Ag}). \\
        \textbf{Part 1:} We first show that $\nabla f(x^k)^\top p^k \rightarrow 0$. To this end, assume that this does not hold. Then there is some $C > 0$ and a strictly increasing, infinite sequence $(k_l)_l \subseteq \N$ with $\nabla f(x^{k_l})^\top p^{k_l} < -C$ for all $l \in \N$. Since the number of selection functions is finite, we can assume w.l.o.g.\ that there is some $i \in I$ with $i \in I_g(x^{k_l + 1})$ for all $l \in \N$. Furthermore, by local Lipschitz continuity of $\nabla f_i$, we can assume w.l.o.g.\ that there is a Lipschitz constant $L > 0$ for $\nabla f_i$ on a superset of $\{ x^{k_l} : l \in \N \}$. By Lemma \ref{lem:lower_bound_Wolfe} we have
        \begin{align*}
            t_{k_l} \geq 
            \underbrace{\frac{1}{L \| p^{k_l} \|^2}}_{\text{(I)}}
            \left( 
            \underbrace{-(1 - c_2) \nabla f(x^{k_l})^\top p^{k_l}}_{\text{(II)}}
            +
            \underbrace{(\nabla f(x^{k_l}) - \nabla f_{i}(x^{k_l}))^\top p^{k_l}}_{\text{(III)}}
            \right).
        \end{align*}
        The fraction (I) is bounded below since $(p^{k_l})_l$ is bounded above due to \ref{enum:B1_2} and \ref{enum:B1_4}. The term (II) is bounded below by $(1 - c_2)C$. The term (III) vanishes by Lemma \ref{lem:common_descent}, since
        \begin{align*}
            \nabla f(x^{k_l}) - \nabla f_{i}(x^{k_l})
            = \nabla f(x^{k_l}) - \nabla f_1(x^{k_l}) - (\nabla f_i(x^{k_l}) - \nabla f_1(x^{k_l})).
        \end{align*}
        Thus, there is some $t_{\text{min}} > 0$ such that $t_{k_l} \geq t_{\text{min}}$ for all $l \in \N$. By the first Wolfe condition \eqref{eq:Wolfe_1}, this implies that
        \begin{align*}
            f(x^{k_l + 1}) - f(x^{k_l}) 
            \leq c_1 t_{k_l} \nabla f(x^{k_l})^\top p^{k_l}
            < - c_1 t_{\text{min}} C
            < 0
            \quad \forall l \in \N,
        \end{align*}
        i.e., the objective value decreases by at least a constant amount infinitely many times. Since $(x^k)_k$ has a limit (by \ref{enum:B1_2}), this contradicts the continuity of $f$. \\
        \textbf{Part 2:} Let $i \in I$. By \ref{enum:B1_3} there is a strictly increasing, infinite sequence $(k_l)_l \subseteq \N$ such that $i \in I_g(x^{k_l})$ for all $l \in \N$. Since $(H_k)_k$ is bounded above by \ref{enum:B1_4}, we can assume w.l.o.g.\ that $(H_{k_l})_l$ converges to some $\bar{H}$. By Part 1, we obtain
        \begin{align*}
            0
            &= \lim_{l \rightarrow \infty} \nabla f(x^{k_l})^\top p^{k_l}
            = -\lim_{l \rightarrow \infty} \nabla f(x^{k_l})^\top H_{k_l} \nabla f(x^{k_l})
            = -\lim_{l \rightarrow \infty} \nabla f_i(x^{k_l})^\top H_{k_l} \nabla f_i(x^{k_l}) \\
            &= -\nabla f_i(\bar{x})^\top \bar{H} \nabla f_i(\bar{x}).
        \end{align*}
        Since $\bar{H}$ is symmetric and positive semidefinite, this implies that $\nabla f_i(\bar{x}) \in \ker(\bar{H})$ (see, e.g., \cite{A2024}, 7.43). Application of Lemma \ref{lem:ker_H_subsequence} completes the proof.
    \end{proof}

    For the function $f : \R^2 \rightarrow \R$, $x \mapsto x_1^2 + |x_2|$ from \cite{GL2018}, it is easy to see that the only point at which both selection functions ($x \mapsto x_1^2 + x_2$ and $x \mapsto x_1^2 - x_2$) are active with linearly dependent gradients is the minimum $x^* = 0 \in \R^2$. However, in general, since convex combinations are a special case of linear combinations, linear dependence of $\nabla f_i(\bar{x})$, $i \in I$, is merely a necessary condition for criticality. One way to guarantee that the vanishing linear combination of the gradients is actually a convex combination is to assume that $f$ has a minimum at which the vanishing convex combination of gradients is ``stable'' in the following sense:
    \begin{assum} \label{assum:A2}
        The function $f : \R^n \rightarrow \R$ satisfies \ref{assum:A1} and
        \begin{enumerate}[leftmargin=1.3cm,label=(A2.\arabic*)]
            \item \label{enum:A2_1} $f$ has a critical point $x^*$ with
                \begin{itemize}
                    \item $\exists \alpha \in \R^m$ with $\alpha_i > 0$ for all $i \in I$ and $\sum_{i = 1}^m \alpha_i \nabla f_i(x^*) = 0$,
                    \item the vectors $\nabla f_i(x^*)$, $i \in I$, are affinely independent,
                \end{itemize}
            \item \label{enum:A2_2} $x^*$ is the unique global minimum and there is some $z \in \R^n \setminus \{ x^* \}$ such that $\calL(z) := \{ x \in \R^n : f(x) \leq f(z) \}$ is bounded.
        \end{enumerate}
    \end{assum}

    The following lemma shows that \ref{enum:A2_1} assures that all vanishing linear combinations of the gradients of active selection functions locally around $x^*$ must be convex combinations:
    \begin{lemma} \label{lem:crit_limit}
        Assume that $f$ satisfies \ref{assum:A1} and \ref{enum:A2_1}. Then there is an open neighborhood $U \subseteq \R^n$ of $x^*$ such if $x \in U$ with $I_e(x) = I = \{1,\dots,m\}$ and $\nabla f_i(x)$, $i \in I$, linearly dependent, then $x$ is a critical point of $f$.
    \end{lemma}
    \begin{proof}
        Assume that this does not hold. Then there is a sequence $(x^l)_l \subseteq \R^n$ with $x^l \rightarrow x^*$, $I_e(x^l) = I$, and $\nabla f_i(x^l)$, $i \in I$, linearly dependent for all $l \in \N$, such that $x^l$ is not critical for any $l \in \N$. Linear dependence implies that there is a sequence $(\beta^l)_l \subseteq \R^m \setminus \{ 0 \}$ with $\sum_{i = 1}^m \beta_i^l \nabla f_i(x^l) = 0$ for all $l \in \N$.
        Assume w.l.o.g.\ (via scaling) that $\| \beta^l \|_\infty := \max_{i \in I} |\beta_i| = 1$ for all $l \in \N$. Then we can assume w.l.o.g.\ that $(\beta^l)_l$ has a limit $\bar{\beta} \in \R^m$ with $\| \bar{\beta} \|_\infty = 1$. By continuity of $\nabla f_i$, $i \in I$, we have $\sum_{i = 1}^m \bar{\beta}_i \nabla f_i(x^*) = 0$. By affine independence of $\nabla f_i(x^*)$, $i \in I$, we must have $\sum_{i = 1}^m \bar{\beta}_i \neq 0$. Let $\beta^* := \bar{\beta}/(\sum_{i = 1}^m \bar{\beta}_i)$. Again using the affine independence, we must have $\beta^* = \alpha$. Now $\alpha_i > 0$ for all $i \in I$ implies that there is some $N \in \N$ such that $\beta_i^l > 0$ for all $i \in I$, $l > N$. This implies that $x^l$ is critical for any $l > N$, which is a contradiction.
    \end{proof}

    To show criticality of the limit $\bar{x}$ via the previous lemma, we have to make sure that $\bar{x}$ lies close enough to the critical point $x^*$ from \ref{enum:A2_1}. Since Alg.\ \ref{algo:QN} is a descent method, we can assure this by assuming that $x^*$ is actually the unique global minimum and that $f$ has compact level sets via \ref{enum:A2_2}. This leads us to the main result of this section: 
    \begin{theorem} \label{thm:criticality}
        Assume that $f$ satisfies \ref{assum:A2}. There is an open neighborhood $U \subseteq \R^n$ of $x^*$ such that if \ref{behav:B1} holds for an initial point $x^0 \in U$, then the limit $\bar{x}$ of $(x^k)_k$ is a critical point of $f$.
    \end{theorem}
    \begin{proof}
        Let $U'$ be the open neighborhood of $x^*$ from Lemma \ref{lem:crit_limit}. By \ref{enum:A2_2} there must be an open neighborhood $U \subseteq \R^n$ of $x^*$ such that $\calL(x^0) \subseteq U'$ for all $x^0 \in U$. Since Alg.\ \ref{algo:QN} is a descent method, $x^0 \in U$ implies $\lim_{k \rightarrow \infty} x^k = \bar{x} \in \calL(x^0) \subseteq U'$. Application of Lemma \ref{lem:intermediate_result} shows that $I_e(\bar{x}) = I$ and the vectors $\nabla f_i(\bar{x})$, $i \in I$, must be linearly dependent. Application of Lemma \ref{lem:crit_limit} completes the proof.
    \end{proof}

\section{Exploration of the piecewise structure} \label{sec:explore_structure}

    For a function $f$ satisfying \ref{assum:A2}, since the gradients of all selection functions are required for the vanishing convex combination of gradients at $x^*$, $x^*$ cannot be a minimum of any continuous selection of a strict subset of selection functions. As a result, any algorithm that is able to minimize such functions must be able to gather information of every selection function during execution. More formally, it must be able to find at least one point from the set $\{ x \in \R^n : i \in I_g(x) \}$ for each $i \in I$. In this section, we show how quasi-Newton methods can achieve this. More precisely, the main result of this section is that if $f$ satisfies \ref{assum:A2}, the quasi-Newton matrices behave in the ``expected way'', and the initial $x^0$ is close enough to $x^*$, then the algorithm visits each of these sets exactly once in the first $m-1$ iterations, i.e.,  
    \begin{align*}
        \bigcup_{k = 0}^{m-1} I_g(x^k) = I = \{1,\dots,m\}.
    \end{align*}
    
    In the following, we first introduce the behavioral assumptions we need to prove this. Clearly, information about all selection functions is only required at points close to the minimum $x^*$. As such, we now focus on the behavior of Alg.\ \ref{algo:QN} when applied to an initial point $x^0$ that is close to $x^*$. By \ref{enum:A2_2} the level set $\calL(x)$ shrinks towards $x^*$ as $x \rightarrow x^*$. Since Alg.\ \ref{algo:QN} is a descent method, it holds $(x^k)_k \subseteq \calL(x^0)$. Thus, for $x^0$ close to $x^*$, all $s^k = x^{k+1} - x^k$ in Alg.\ \ref{algo:QN} must be small. As discussed at the beginning of Section \ref{sec:limit_critical}, if $i \in I_g(x^k)$ and $j \in I_g(x^{k+1})$ with $i \neq j$, then $s^k$ being small means that the secant equation \eqref{eq:secant_equation} forces an eigenvalue of $H_{k+1}$ to be small. For simplicity, assume that $H_0$ is the identity matrix $\id$. Then we expect that after $k \in \{0,\dots,m-1\}$ such updates to $H_0$, the resulting $H_k$ has at most $k$ small eigenvalues. The following example suggests that for the BFGS method, applied to the function already considered in Example \ref{example:criticality}, this is indeed the case:
    \begin{example} \label{example:exploration}
        Consider the function $f$ from Example \ref{example:criticality} for $n = 10$ and $m = 6$. We generate a random instance of this function and apply $m-1 = 5$ iterations of the BFGS method with $c_1 = 10^{-4}$, $c_2 = 0.5$, and $H_0 = \id$ to each of $100$ different initial points. The initial points are chosen randomly, but with specified distances to $x^* = 0$, such that the values $\log_{10}(\| x^0 - x^* \|)$ are equidistant in $[-30,2]$. (For details on the random generation, see the corresponding code.) The results are computed with $500$ significant digits via Matlab's variable-precision arithmetic.
        \begin{figure}
            \centering
            \parbox[b]{0.49\textwidth}{
                \centering 
                \includegraphics[width=0.45\textwidth]{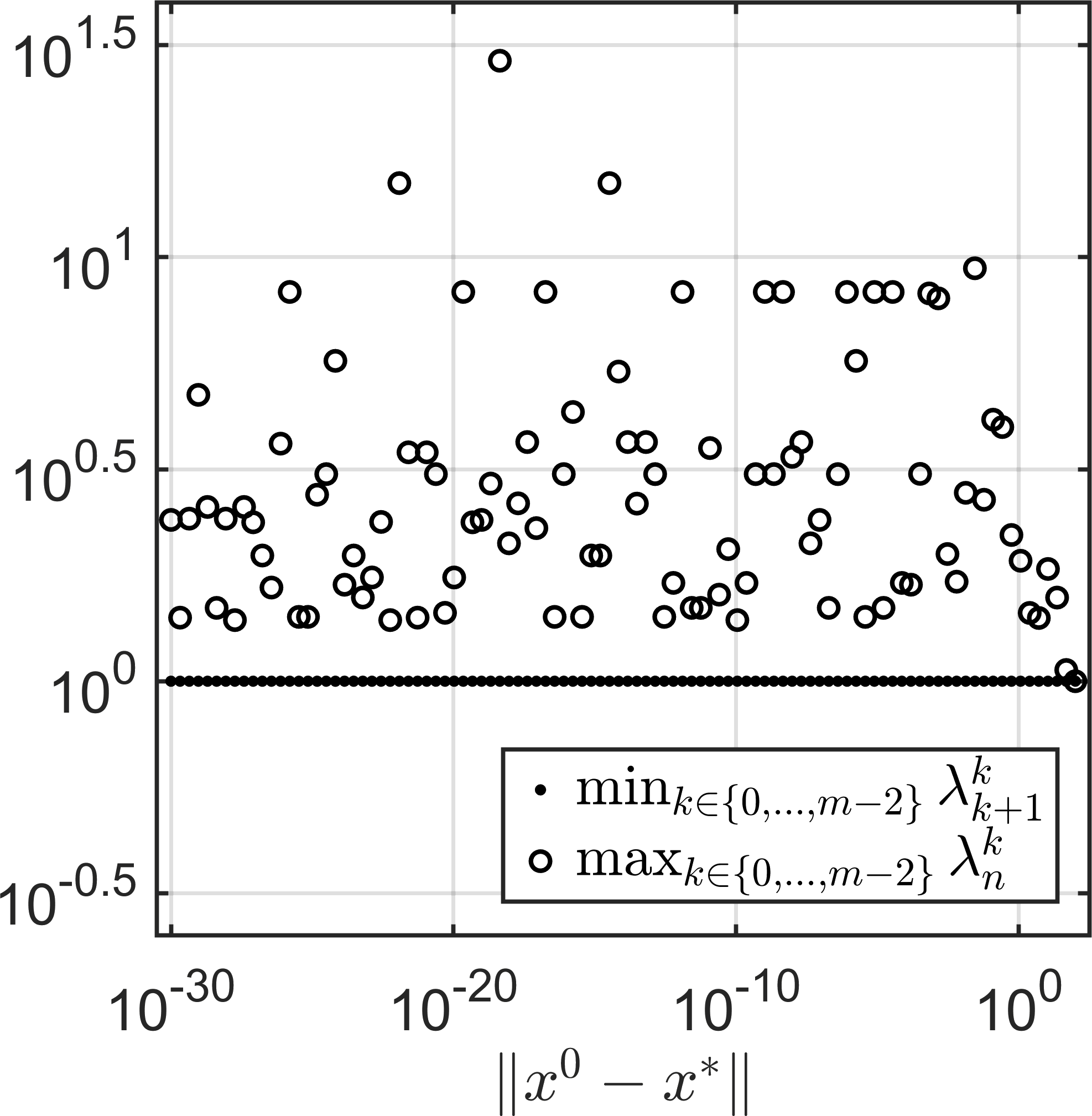}\\
                (a)
    		}
            \parbox[b]{0.49\textwidth}{
                \centering 
                \includegraphics[width=0.45\textwidth]{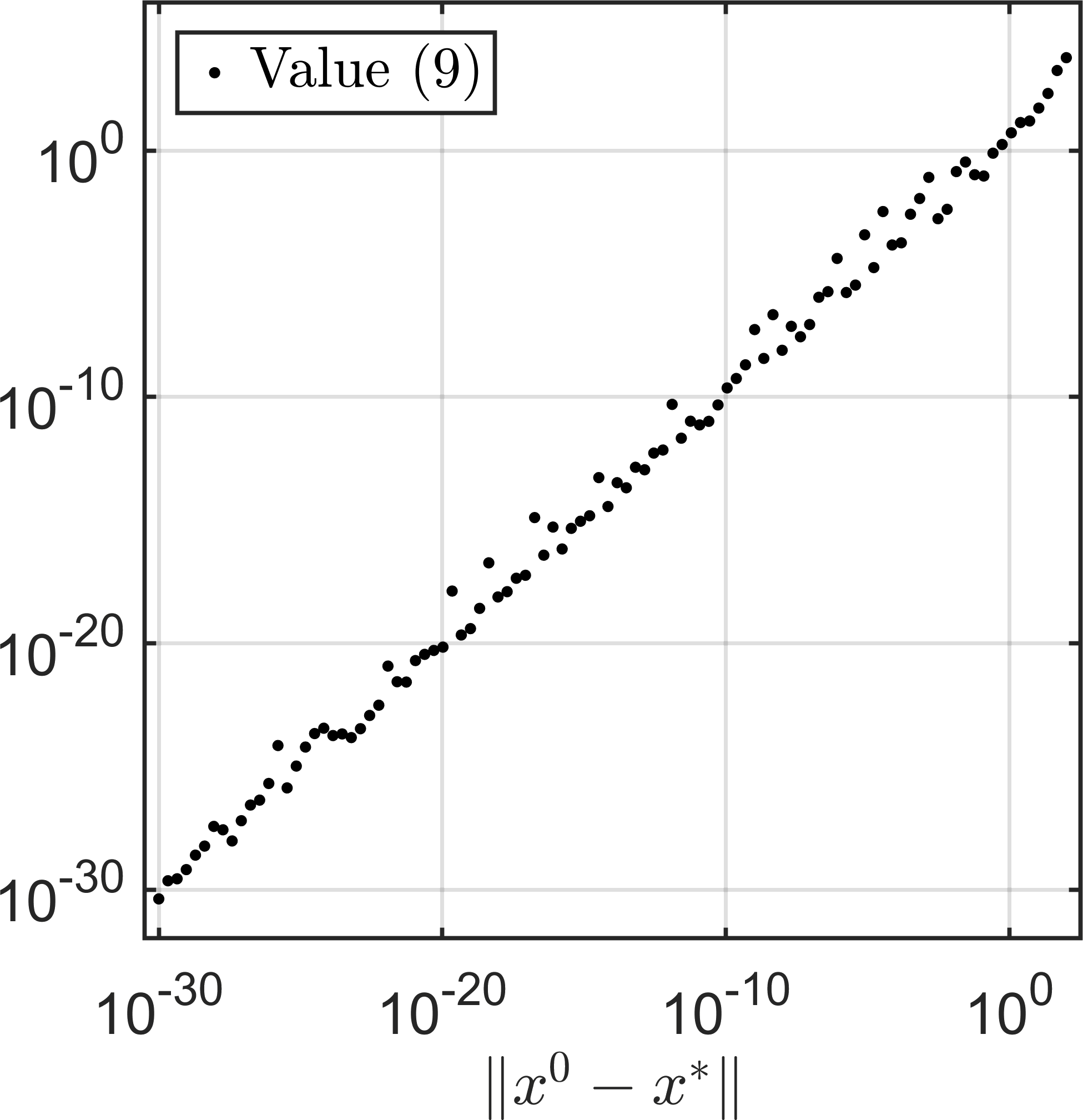}\\
                (b)
    		}
            \caption{(a) The relevant eigenvalues of $(H_k)_k$ in the first $m-1$ iterations in Example \ref{example:exploration} for initial points closer and closer to $x^*$. (Note that the dots are all exactly on the value $10^0 = 1$.) (b) The value \eqref{eq:theta_numerically} for the same initial points.}
            \label{fig:example_exploration}
        \end{figure}
        For each of these runs, Fig.\ \ref{fig:example_exploration}(a) shows, for $k \in \{0,\dots,m-2\}$, the smallest $(k+1)$-th eigenvalue of any $H_k$ (i.e., $\min_{k \in \{0,\dots,m-2\}} \lambda_{k+1}^k$) and the largest eigenvalue of any $H_k$ (i.e., $\max_{k \in \{0,\dots,m-2\}} \lambda_n^k$), plotted against the distance of the corresponding initial point to $x^*$. These values appear to be bounded below and above, respectively. Fig.\ \ref{fig:example_exploration}(b) shows the value
        \begin{align} \label{eq:theta_numerically}
            \max_{k' \in \{1,\dots,m-2\}}
            \max_{k \in \{0,\dots,k'-1\}}
            \| H_{k'} y^k \|
        \end{align}
        for every run, which appears to vanish.
    \end{example}
    Example \ref{fig:example_exploration} also suggests that the largest eigenvalue of $H_k$ is bounded over all runs. Furthermore, crucially, it suggests that the value \eqref{eq:theta_numerically} vanishes as $x^0$ approaches $x^*$. For $k = k'-1$, the secant equation \eqref{eq:secant_equation} in iteration $k'$ yields $H_{k'} y^{k} = H_{k'} y^{k'-1} = s^{k'-1}$, so $H_{k'} y^{k}$ vanishes as $x^0$ approaches $x^*$ (by \ref{enum:A2_2}, as discussed above). Now \eqref{eq:theta_numerically} vanishing means that the same is true for any $k \in \{0,\dots,k'-1\}$. This suggests that, similar to the observation in Example \ref{example:criticality}, the BFGS update causes the quasi-Newton matrix to ``memorize'' previous secant equations.
    
    As in Section \ref{sec:limit_critical}, we now rephrase the above observation as a formal assumption on the behavior of $(H_k)_k$. We do this by considering a sequence $(x^{l,0})_l \subseteq \R^n$ of initial points for Alg.\ \ref{algo:QN} with $\lim_{l \rightarrow \infty} x^{l,0} = x^*$. For $l \in \N$, we denote the sequences generated by the algorithm with initial point $x^{l,0}$ by $(x^{l,k})_k$, $(H_{l,k})_k$, $(s^{l,k})_k$, $(y^{l,k})_k$, $(p^{l,k})_k$, and $(t_{l,k})_k$, respectively, and the sorted eigenvalues of $H_{l,k}$ (in increasing order) by $\lambda_j^{l,k}$, $j \in \{1,\dots,n\}$.
    \begin{behav} \label{behav:B2}
        For a function satisfying \ref{assum:A2} and for a sequence $(x^{l,0})_l \subseteq \R^n$ of initial points with $\lim_{l \rightarrow \infty} x^{l,0} = x^*$, assume that there are $\sigma_L, \sigma_U \in \R^{>0}$ such that for each $l \in \N$, Alg.\ \ref{algo:QN}, with fixed parameters $c_1$ and $c_2$, does not stop in the first $m-1$ iterations, and it holds
        \begin{enumerate}[leftmargin=1.25cm,label=(B2.\arabic*)]
            \item \label{enum:B2_1}
            \begin{align*}
                \lambda_j^{l,k} \in [\sigma_L, \sigma_U] \quad \forall j \in \{k+1,\dots,n\}, k \in \{0,\dots,m-2\},
            \end{align*}
            \item \label{enum:B2_2}
            \begin{align*}
                \lim_{l \rightarrow \infty} H_{l,k'} y^{l,k} = 0
                \quad \forall k' \in \{1,\dots,m-2\}, k \in \{0,\dots,k'-1\}.
            \end{align*}
        \end{enumerate}
    \end{behav}

    To prove the main result of this section, we require the following technical lemma:

    \begin{lemma} \label{lem:aff_lin_indep_and_A_g_singleton}
        Assume that $f$ satisfies \ref{assum:A1}.
        \begin{enumerate}[label=(\alph*)]
            \item There is an open neighborhood $U \subseteq \R^n$ of $x^*$ such that $|I_g(x)| = 1$ for all $x \in U \setminus \Omega$.
            \item If $f$ satisfies \ref{enum:A2_1}, then for any $i^* \in I$, the gradients $\nabla f_i(x^*)$, $i \in I \setminus \{ i^* \}$, are linearly independent.
        \end{enumerate}
    \end{lemma}
    \begin{proof}
        \textbf{(a)} By continuity of $\nabla f_i$, $i \in I$, there is an open neighborhood $U \subseteq \R^n$ of $x^*$ such that the vectors $\nabla f_i(x)$, $i \in I$, are affinely independent for all $x \in U$. In particular, for $x \in U \setminus \Omega$, $i' \in I_g(x)$, and all $i'' \in I \setminus \{ i' \}$, it holds $\nabla f(x) = \nabla  f_{i'}(x) \neq \nabla f_{i''}(x)$, which completes the proof. \\
        \textbf{(b)} Assume that this does not hold for some $i^* \in I$. Then there are $\beta_i \in \R$, $i \in I \setminus \{ i^* \}$, such that $\sum_{i = 1, i \neq i^*}^{m-1} \beta_i \nabla f_i(x^*) = 0$ and $\beta_i \neq 0$ for some $i \in I \setminus \{ i^* \}$. Define $\beta_{i^*} := 0$ and $\beta := (\beta_1,\dots,\beta_m)^\top$. By affine independence, we must have $\sum_{i = 1}^m \beta_i \neq 0$. Let $\beta^* := \beta / (\sum_{i = 1}^m \beta_i)$. Again using affine independence, we must have $\beta^* = \alpha$. But this is a contradiction, since $0 = \beta^*_{i^*} = \alpha_{i^*} > 0$ by \ref{enum:A2_1}.
    \end{proof}

    Lemma \ref{lem:aff_lin_indep_and_A_g_singleton}(a) implies that close to $x^*$, the gradient at every iterate belongs to a unique selection function. This allows us to consider the order in which the algorithm discovers the selection functions, which is the starting point for the proof of our second main result. The remainder of the proof is similar to the proof of Lemma \ref{lem:intermediate_result}, but with taking the limit $l \rightarrow \infty$ instead of $k \rightarrow \infty$:
    \begin{theorem} \label{thm:exploration}
        Assume that $f$ satisfies \ref{assum:A2} and let $(x^{l,0})_l$ be a sequence as in \ref{behav:B2}. Then there is some $N > 0$ such that for all $l > N$, it holds
        \begin{align} \label{eq:learning_property}
            \bigcup_{k = 0}^{m-1} I_g(x^{l,k}) = I = \{1,\dots,m\}.
        \end{align}
    \end{theorem}
    \begin{proof}
        Assume that this does not hold. Then there must be infinitely many $l \in \N$ for which \eqref{eq:learning_property} is violated. Assume w.l.o.g.\ that \eqref{eq:learning_property} is violated for every $l \in \N$. \\
        \textbf{Part 1:} By Lemma \ref{lem:aff_lin_indep_and_A_g_singleton}(a) and since the algorithm did not break down, we can assume w.l.o.g.\ that $|I_g(x^{l,k})| = 1$ for all $l \in \N$, $k \in \{0,\dots,m-1\}$. Furthermore, since the number of selection functions is finite, we can assume w.l.o.g.\ that the order in which the selection functions are encountered is the same for any $l \in \N$, i.e., we can assume that the vector $(i_0, \dots, i_{m-1})$ with $I_g(x^{l,k}) = \{ i_k \}$ for $k \in \{0,\dots,m-1\}$ is the same for any $l \in \N$. Since we assumed that \eqref{eq:learning_property} is violated, there must be a smallest $k^* \in \{0,\dots,m-2\}$ such that $i_{k^*+1} \in \{i_0, \dots, i_{k^*}\}$. Let $k^\circ \in \{0,\dots,k^*\}$ with $i_{k^*+1} = i_{k^\circ}$. (Then $f_{i_{k^*+1}}$ is the first selection function that is encountered twice along $(x^{l,k})_k$, at iterations $k^\circ$ and $k^*+1$.) \\
        \textbf{Part 2:} For any $k \in \{0, \dots, k^* - 1 \}$, we can write
        \begin{align*}
            \nabla f(x^{l,k^*}) - \nabla f(x^{l,k})
            &= ( \nabla f(x^{l,k^*}) - \nabla f(x^{l,k^*- 1}) ) + ( \nabla f(x^{l,k^*- 1}) - \nabla f(x^{l,k}) ) \\
            &= y^{l,k^*-1}  + ( \nabla f(x^{l,k^*- 1}) - \nabla f(x^{l,k}) )
            = y^{l,k^*-1} + \dots + y^{l,k},
        \end{align*}
        so \ref{enum:B2_2} implies
        \begin{align} \label{eq:proof_thm_learning_3}
            \lim_{l \rightarrow \infty} H_{l,k^*} (\nabla f(x^{l,k^*}) - \nabla f(x^{l,k})) = 0
            \quad \forall k \in \{0, \dots, k^* - 1 \}.
        \end{align}
        Furthermore, for any $k \in \{0, \dots, k^* - 1 \}$, we have
        \begin{equation} \label{eq:proof_thm_learning_2tmp}
            \begin{aligned}
                &H_{l,k^*} (\nabla f(x^{l,k^*}) - \nabla f_{i_{k}}(x^{l,k^*})) \\
                &= H_{l,k^*} (\nabla f(x^{l,k^*}) - \nabla f(x^{l,k})) + H_{l,k^*} (\nabla f_{i_k}(x^{l,k}) - \nabla f_{i_{k}}(x^{l,k^*})).
            \end{aligned}
        \end{equation}
        For $l \rightarrow \infty$, the first summand on the right-hand side of \eqref{eq:proof_thm_learning_2tmp} vanishes by \eqref{eq:proof_thm_learning_3}. The second summand vanishes by boundedness of $(H_{l,k})_l$ (cf.\ \ref{enum:B2_1}), continuity of $\nabla f_{i_k}$ (cf.\ \ref{assum:A1}), and since $\lim_{l \rightarrow \infty} x^{l,k} = x^*$ for all $k \in \N$ (cf.\ \ref{enum:A2_2}). This means that
        \begin{align} \label{eq:proof_thm_learning_2}
            \lim_{l \rightarrow \infty} H_{l,k^*} (\nabla f(x^{l,k^*}) - \nabla f_{i_{k}}(x^{l,k^*})) = 0 \quad \forall k \in \{0, \dots, k^*-1\}.
        \end{align}
        For $k = k^*$, \eqref{eq:proof_thm_learning_2} also holds trivially since $\nabla f(x^{l,k^*}) = \nabla f_{i_{k^*}}(x^{l,k^*})$. \\
        \textbf{Part 3:} By construction it holds $I_g(x^{l,k^*}) = \{ i_{k^*} \}$ and $I_g(x^{l,k^*+1}) = \{ i_{k^\circ} \}$, so
        application of Lemma \ref{lem:lower_bound_Wolfe} yields
        \begin{align*}
            t_{l,k^*} 
            \geq 
            \underbrace{\frac{1}{L \| p^{l,k^*} \|^2}}_{\text{(I)}}
            \left(
            \underbrace{-(1 - c_2) \nabla f(x^{l,k^*})^\top p^{l,k^*}}_{\text{(II)}}
            +
            \underbrace{(\nabla f(x^{l,k^*}) - \nabla f_{i_{k^\circ}}(x^{l,k^*}))^\top p^{l,k^*}}_{\text{(III)}}
            \right).
        \end{align*} 
        for all $l \in \N$ (where $L$ is a common Lipschitz constant for the gradients of all selection functions locally around $x^*$, cf.\ \ref{assum:A1}). The fraction (I) is bounded below by boundedness of $(H_{l,k})_l$ (cf.\ \ref{enum:B2_1}). The term (III) vanishes for $l \rightarrow \infty$ by \eqref{eq:proof_thm_learning_2}, since $p^{l,k^*} = -H_{l,k^*} \nabla f(x^{l,k^*})$ and $k^\circ \in \{0,\dots,k^*\}$. Regarding (II), if there would be some $C > 0$ with $\nabla f(x^{l,k^*})^\top p^{l,k^*} < -C$ for infinitely many $l \in \N$, then there would be some $t_{\text{min}} > 0$ such that $t_{l,k^*} \geq t_{\text{min}}$ for infinitely many $l \in \N$. By the first Wolfe condition \eqref{eq:Wolfe_1}, this would mean that
        \begin{align*}
            f(x^{l,k^*+1}) - f(x^{l,k^*})
            \leq c_1 t_{l,k^*} \nabla f(x^{l,k^*})^\top p^{l,k^*}
            < -c_1 t_{\text{min}} C
        \end{align*}
        for such $l$. This is a contradiction, since
        \begin{align*}
            0 > f(x^{l,k^*+1}) - f(x^{l,k^*}) \geq f(x^*) - f(x^{l,0}) \rightarrow 0.
        \end{align*}
        Thus, the term (II) must vanish as well, i.e., $\lim_{l \rightarrow \infty} \nabla f(x^{l,k^*})^\top p^{l,k^*} = 0$. \\
        \textbf{Part 4:} By the upper bound in \ref{enum:B2_1}, we can assume w.l.o.g.\ that $(H_{l,k})_l$ has a limit $\bar{H}_{k}$ for any $k \in \{0, \dots, m-2\}$. 
        By the lower bound in \ref{enum:B2_1} and continuity of eigenvalues, we have $\dim(\ker(\bar{H}_{k^*})) \leq k^*$.
        By \eqref{eq:proof_thm_learning_2} it holds
        \begin{align*}
            \bar{H}_{k^*} (\nabla f_{i_{k^*}}(x^*) - \nabla f_{i_{k}}(x^*))
            = 0 \quad \forall k \in \{0, \dots, k^* - 1 \}.
        \end{align*}
        By \ref{enum:A2_1} all vectors $\nabla f_{i_{k^*}}(x^*) - \nabla f_{i_k}(x^*)$, $k \in \{0, \dots, k^* - 1 \}$, are linearly independent. (Recall that $k^*$ is the \emph{smallest} index for which $f_{i_{k^*+1}}$ is encountered twice.) This implies $\dim(\ker(\bar{H}_{k^*})) = k^*$ and
        \begin{align} \label{eq:proof_thm_learning_5}
            \ker(\bar{H}_{k^*})
            = \spn(\{ \nabla f_{i_{k^*}}(x^*) - \nabla f_{i_k}(x^*) : k \in \{0, \dots, k^* - 1 \} \}).
        \end{align}
        By Part 3 we must have
        \begin{align*}
            0
            &= \lim_{l \rightarrow \infty} \nabla f(x^{l,k^*})^\top p^{l,k^*}
            = \lim_{l \rightarrow \infty} -\nabla f(x^{l,k^*})^\top H_{l,k^*} \nabla f(x^{l,k^*}) \\
            &= \lim_{l \rightarrow \infty} -\nabla f_{i_{k^*}}(x^{l,k^*})^\top H_{l,k^*} \nabla f_{i_{k^*}}(x^{l,k^*}) 
            = -\nabla f_{i_{k^*}}(x^*)^\top \bar{H}_{k^*} \nabla f_{i_{k^*}}(x^*).
        \end{align*}
        Since $\bar{H}_{k^*}$ is positive semi-definite, this shows that $\nabla f_{i_{k^*}}(x^*) \in \ker(\bar{H}_{k^*})$ (see, e.g., \cite{A2024}, 7.43). By \eqref{eq:proof_thm_learning_5} this implies $\nabla f_{i_k}(x^*) \in \ker(\bar{H}_{k^*})$ for all $k \in \{0,\dots,k^*-1\}$. Since $\dim(\ker(\bar{H}_{k^*})) = k^*$, the $k^* + 1$ vectors $\nabla f_{i_k}(x^*) \in \ker(\bar{H}_{k^*})$, $k \in \{0,\dots,k^*\}$, must be linearly dependent. This is a contradiction to Lemma \ref{lem:aff_lin_indep_and_A_g_singleton}(b), since $k^* + 1 \leq m - 1$ by construction.        
    \end{proof}

    The behavior described in Theorem \ref{thm:exploration} can be nicely observed when considering quasi-Newton methods with restarts, where the quasi-Newton matrix $H_k$ is periodically reset to the initial $H_0$. 
    This technique is typically employed by conjugate gradient methods to erase old information from the algorithm (see, e.g., \cite{NW2006}, Section 5.2), and has a similar effect here, in that it forces Alg.\ \ref{algo:QN} to ``relearn'' the piecewise structure of the objective. By Theorem \ref{thm:exploration}, close to $x^*$, exactly $m-1$ iterations are required to detect all selection functions of a function satisfying \ref{assum:A2} (since one selection function is already known from the initial point). 
    In the $m$-th iteration, the search direction then yields (sufficient) decrease for all selection functions at the same time, which allows for a significant decrease of the objective value. The following example visualizes this behavior, and even suggests that the BFGS method with restarts every $m$ iterations still converges:
    \begin{example} \label{example:restarts}
        Consider the function $f$ from Example \ref{example:criticality} for $n = 100$ and $m = 80$. We randomly generate an instance of this function and an initial point $x^0 \in \R^n$. (For details on the random generation, see the corresponding code.) We apply $18 \cdot 80 = 1440$ iterations of the BFGS method with restarts every $m$ iterations and $H_0 = \id$. (To be precise, when $k$ is a multiple of $m$, then we set $H_k = H_0 = \id$.) For the Wolfe step length, we use $c_1 = 0.5$ and $c_2 = 0.75$. In contrast to the previous examples, we use Matlab's default accuracy for this experiment.
        \begin{figure}
            \centering
            \parbox[b]{0.32\textwidth}{
                \centering 
                \includegraphics[width=0.32\textwidth]{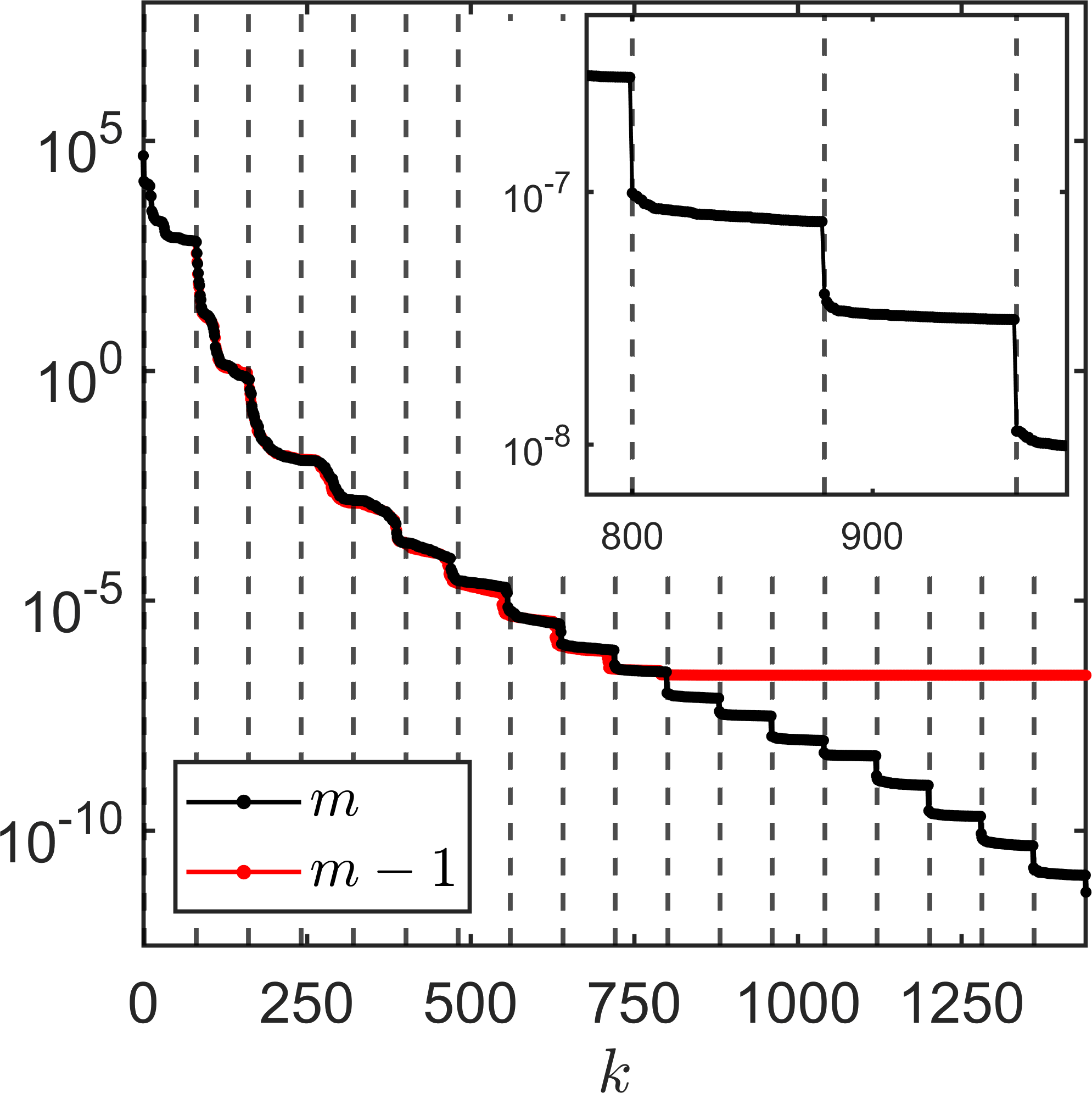}\\
                (a)
    		}
            \parbox[b]{0.32\textwidth}{
                \centering 
                \includegraphics[width=0.30\textwidth]{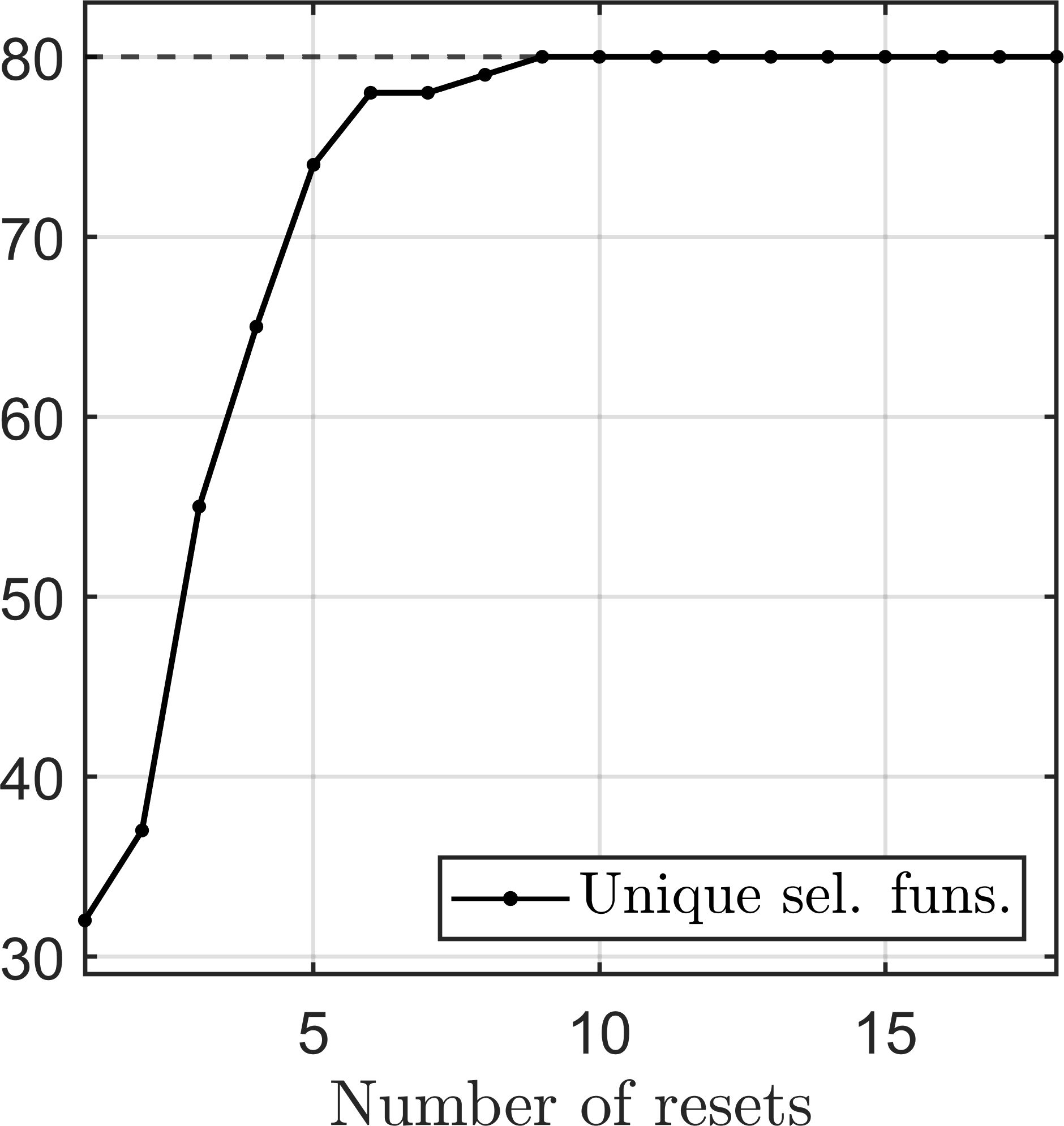}\\
                (b)
    		}
            \parbox[b]{0.32\textwidth}{
                \centering 
                \includegraphics[width=0.32\textwidth]{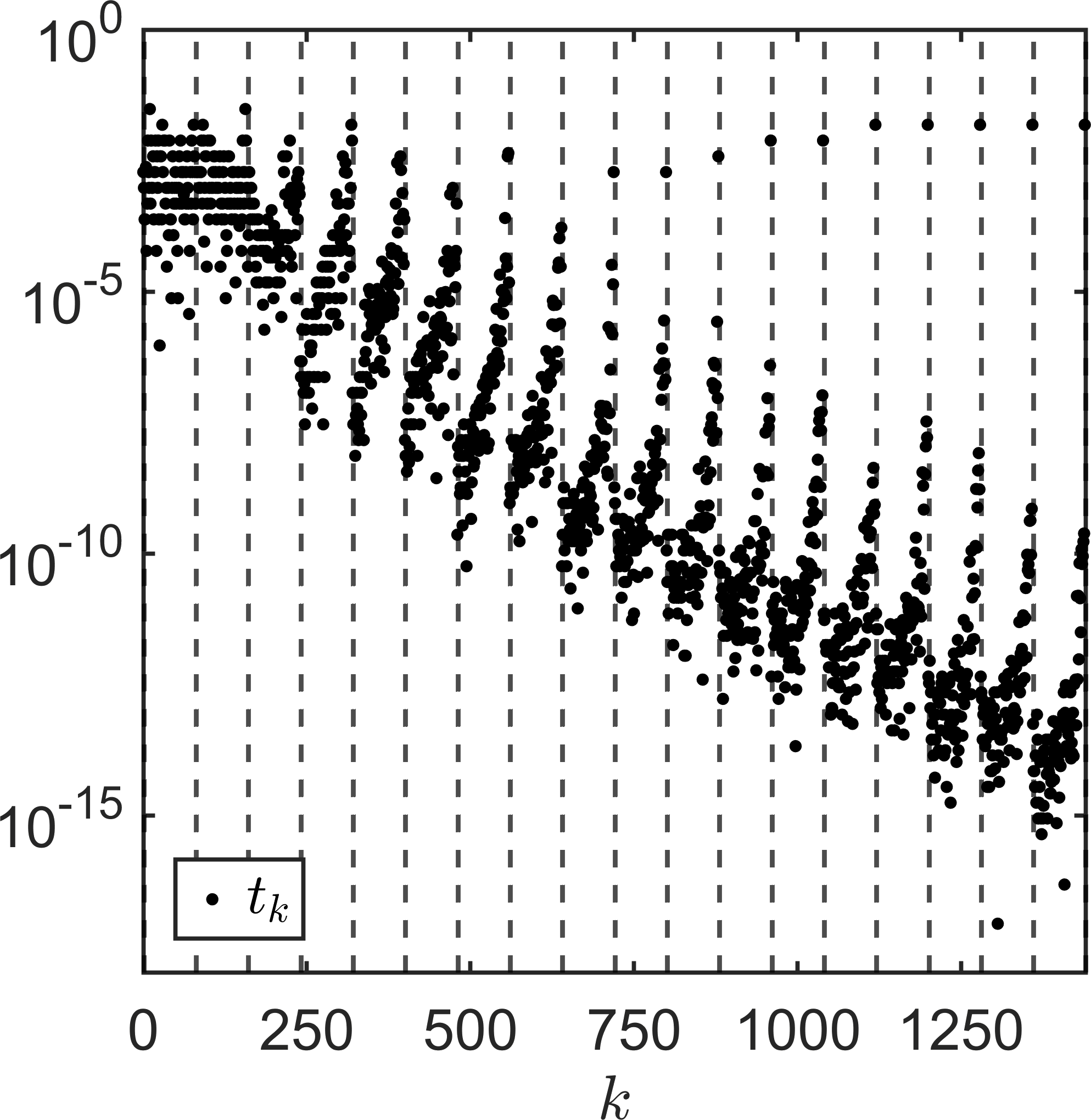}\\
                (c)
    		}
            \caption{(a) The distance of $f(x^k)$ to $f(x^*)$ in Example \ref{example:restarts}, with restarts every $m$ (black) or $m-1$ (red) iterations. (b) The number of unique selection functions encountered between restarts. (c) The step lengths $(t_k)_k$.}
            \label{fig:example_restarts}
        \end{figure}
        Fig.\ \ref{fig:example_restarts}(a) shows, in black, the distance of the objective values of the generated sequence to the optimal value, with dashed, vertical lines highlighting restarts. We see that, as $k$ increases, the objective value decreases in a stepwise fashion.
        Also shown, in red, is the same data, except that the BFGS method is restarted every $m-1$ instead of every $m$ iterations. We see that the method gets stuck and does not converge when restarts are too frequent.
        Fig.\ \ref{fig:example_restarts}(b) shows the number of unique selection functions encountered between restarts, showing that the algorithm successfully discovers all $m = 80$ selection functions after every restart as $k$ increases.
        Finally, Fig.\ \ref{fig:example_restarts}(c) shows the sequence of step lengths $(t_k)_k$. We see that they vanish, except for the final ones at the end of the restart periods.
    \end{example}

    Example \ref{example:restarts} shows that for $x^k$ close to $x^*$, the first $m-1$ iterations after a restart yield almost no decrease, which is reflected in the step lengths being small. (This becomes more pronounced the larger $c_1 \in (0,1)$, which is why we used a larger parameter here compared to Example \ref{example:criticality} and Example \ref{example:exploration}.) Only in the $m$-th iteration, a significant decrease is achieved. This behavior is similar to the gathering of subgradient information in bundle methods: if the current model in these methods is insufficient, then null steps are employed, which are steps that only gather new subgradients to enrich the model and do not actually decrease the objective value. It is also similar to the deterministic gradient sampling strategy \cite{MY2012} (or \cite{G2024}, Section 2.2), where subgradients from the Goldstein $\eps$-subdifferential are iteratively gathered to compute a stabilized descent direction.

\section{Discussion and outlook} \label{sec:outlook}

    In this article, we analyzed the convergence of quasi-Newton methods for piecewise differentiable functions. We showed that when assuming that the quasi-Newton matrix $(H_k)_k$ behaves as it typically does in numerical experiments (specifically for the BFGS method), then convergence results can be derived in a relatively simple way. The first main result (Theorem \ref{thm:criticality}) is the criticality of the limit for a class of well-behaved piecewise differentiable functions (cf.\ \ref{assum:A2}). The second main result (Theorem \ref{thm:exploration}) shows how quasi-Newton methods are able to explore the piecewise structure of such functions locally around the minimum.

    There are several open questions and possibilities for future research:
    \begin{itemize}
        \item The obvious question is whether it is possible to prove that for the BFGS method, the assumption \ref{behav:B1}, specifically \ref{enum:B1_4} and \ref{enum:B1_5}, actually holds in some general setting. The way it is stated in this article, we believe that this is not possible: in case $m = 1$ (i.e., in the smooth case), \ref{enum:B1_4} implies that the condition number of $H_k$ is bounded for $k \in \N$ which, in turn, implies that the angle between the search direction $p^k$ and the gradient $\nabla f(x^k)$ is bounded away from $90^\circ$.
        (For this case, our results essentially reduce to a special case of Zoutendijk's theorem, cf.\ the discussion on p.\ $40$ in \cite{NW2006}.)
        According to \cite{XBN2020}, p.\ $184$, it is not possible to find a bound for this condition number without already knowing that the sequence $(x^k)_k$ converges to a minimum.
        However, Theorem 2.1 in \cite{BN1989} shows that under weak assumptions, the above angle is bounded away from $90^\circ$ for at least a constant fraction of iterations, which is sufficient to prove convergence in the smooth case. If one can show that this still holds in some generalized sense for the nonsmooth case, and if one can generalize the results of Section \ref{sec:limit_critical} to only require \ref{enum:B1_4} and \ref{enum:B1_5} for a constant fraction of iterations, then a proof of convergence for the nonsmooth case may be achievable.
        \item By Section \ref{sec:explore_structure}, performing a small number of iterations of the BFGS method could be seen as a mechanism for exploring the nonsmooth structure of the objective function close to the minimum. As such, it could be inserted into other solution methods, like bundle or gradient sampling methods, as a (heuristic) way to gather subgradients for these methods without the need to solve any (linear or quadratic) subproblems. Note, however, that the step lengths become relatively small (cf.\ Figure \ref{fig:example_restarts}(c)), which may cause numerical issues.
        \item Theorem \ref{thm:exploration} and Example \ref{example:restarts} suggest that knowledge of the previous $m-1$ iterations is sufficient for the BFGS method to achieve convergence. Limiting the information stored in the quasi-Newton matrix in this way is similar to the idea of limited-memory BFGS (L-BFGS) methods (see, e.g., \cite{NW2006}, Section 7.2), where the quasi-Newton matrix is computed from a (typically small) fixed number of recent update pairs $(s^k,y^k)$. In \cite{AO2021}, the behavior of L-BFGS methods on nonsmooth functions was analyzed, with the result that they perform poorly compared to the full BFGS method. The results in Section \ref{sec:explore_structure} may be related to this, as we lose the convergence in Example \ref{example:restarts} already when restarting every $m-1$ instead of every $m$ iterations.
        \item The affine independence in \ref{enum:A2_1} is a strong assumption, as it is violated as soon as $m > n+1$. (For example, even for functions as simple as the $\ell_1$-norm on $\R^2$, this assumption is violated.) For proving a result like Theorem \ref{thm:criticality} under weaker assumptions, one likely has to improve Lemma \ref{lem:intermediate_result} by deriving a stronger property than \eqref{eq:intermediate_result}. For Theorem \ref{thm:exploration}, it may be possible to prove, under weaker assumptions, that the first $\dim(\calN(\bar{x}))+1$ selection functions (cf.\ \eqref{eq:def_calN}) that are encountered along $(x^{l,k})_k$ (for large $l$) have affinely independent gradients.
    \end{itemize}

\vspace{10pt}

\setlength{\bibsep}{0pt plus 0.3ex}

\noindent \textbf{Acknowledgements.} \quad This research was funded by Deutsche Forschungsgemeinschaft (DFG, German Research Foundation) – Projektnummer 545166481.

\bibliography{references}

\end{document}